\documentclass[12pt]{amsart}

\newtheorem{theorem}{Theorem}[section]

\newtheorem{lemma}[theorem]{Lemma}

\DeclareMathOperator{\BB}{{BB}}

\DeclareMathOperator{\cl}{cl}
\DeclareMathOperator{\PG}{PG}
\DeclareMathOperator{\AG}{AG}

\newcommand{\del}{\setminus}
\newcommand{\con}{/}

\begin{document}

\sloppy

\title{A geometric version of the Andr\'asfai-Erd\H os-S\' os Theorem}

\author{Jim Geelen}
\email{jim.geelen@uwaterloo.ca}
\address{Department of Combinatorics and Optimization,
University of Waterloo, Canada}
\thanks{This research is was partially supported by 
a grant from the Office of Naval Research [N00014-10-1-0851].}

\date{\today}

\begin{abstract}
For each odd integer $k\ge 5$, we prove that,
if $M$ is a simple rank-$r$ binary matroid with no odd circuit of length
less than $k$ and with $|M| > k 2^{r-k+1}$, 
then $M$ is isomorphic to a restriction of the 
rank-$r$ binary affine geometry; this bound is tight for all $r\ge k-1$.
We use this to give a simpler proof of the following
result of Govaerts and Storme:
for each integer $n\ge 2$,  
if $M$ is a simple rank-$r$ binary matroid with no 
$\PG(n-1,2)$-restriction and with
$|M| > \left(1-\frac{11}{2^{n+2}}\right) 2^r$,
then $M$ has critical number at most $n-1$.
That result is a geometric analogue of a theorem 
of Andr\' asfai, Erd\H os and S\' os in extremal graph theory.
\end{abstract}

\maketitle

\section{Introduction}
Our main result is:
\begin{theorem}\label{main}
For each odd integer $k\ge 5$ and each integer $r\ge k-1$, 
if $M$ is a simple rank-$r$ binary matroid with no odd circuit of length
less than $k$ and with $|M| > \frac{k}{2^{k-1}} 2^r$, 
then $M$ is isomorphic to a restriction of the 
rank-$r$ binary affine geometry.
\end{theorem}
Examples showing that the bound is tight
are given in Section~\ref{examples}.
We will prove Theorem~\ref{main} in Section~\ref{proofs}.
In the remainder of this introduction we discuss the motivation.

We will call a matroid {\em $N$-free} if it has
no restriction isomorphic to $N$.
Bose and Burton~[\ref{bb}] proved the following theorem.
\begin{theorem}[Bose-Burton Theorem]\label{bbthm}
For all integers $r$ and $n$ with $r\ge n\ge 2$,
if $M$ is a simple rank-$r$ $\PG(n-1,2)$-free binary matroid,
then $|M| \le \left(1-\frac{1}{2^{n-1}}\right)2^r$.
\end{theorem}

Note that, if $F$ is a rank-$(r-n+1)$ flat in
$\PG(r-1,2)$, then the matroid $M=\PG(r-1,2)\del F$
attains equality in the Bose-Burton Theorem.
We will denote $\PG(r-1,2)\del F$ by $\BB(r,n-1)$.
Note that $\BB(r,1)$ is isomorphic to the 
affine geometry $\AG(r-1,2)$.
Bose and Burton proved that $\BB(r,n-1)$ is the
only matroid that attains equality in their theorem,
but the following result is considerably stronger.

The {\em critical number} of a simple rank-$r$
binary matroid $M$ is equal to the minimum integer
$c$ such that $M$ is isomorphic to a restriction of $\BB(r,c)$.
Equivalently, $c$ is the minimum number of cocycles
of $M$ required to cover $E(M)$. (Here by a {\em cocycle}
we mean a disjoint-union of cocircuits.)
The following result was proved by Govaerts and Storme~[\ref{gs}];
it is analogous to a theorem in extremal graph theory 
due to Andr\'asfai, Erd\H os and S\' os~[\ref{aes}]. We will 
review the related work in graph theory in the next section. 
\begin{theorem}[Geometric Andr\' asfai-Erd\H os-S\' os Theorem] \label{gsthm}
Let  $n\ge 2$ be an integer and let $\epsilon = \frac{3}{2^{n+2}}$.  
Then, for each integer $r\ge n+2$,  
if $M$ is a simple rank-$r$ $\PG(n-1,2)$-free binary matroid
with $|M| > \left(1-\frac{1}{2^{n-1}} - \epsilon\right) 2^r$,
then $M$ has critical number at most $n-1$.
\end{theorem}

A simple binary matroid is called {\em affine} if it is
isomorphic to a restriction of a binary affine geometry; 
that is, the ground set is itself a cocycle.
Note that the $n=2$ instance of the Geometric 
Andr\'asfai-Erd\H os-S\' os
Theorem is the same as the $k=5$ instance of Theorem~\ref{main},
and both equivalent to the following result.
\begin{theorem}\label{tfthm}
For each integer $r\ge 4$,
if $M$ is a simple rank-$r$ triangle-free binary matroid
with $|M| > \frac{5}{16} 2^r$, then $M$ is affine.
\end{theorem}

Govaerts and Storme prove the 
Geometric Andr\' asfai-Erd\H os-S\' os Theorem
by induction on $n$. The induction follows an existing
method introduced by Beutelspacher~[\ref{beut}],
but the base case (Theorem~\ref{tfthm})
requires some work. Our proof
of Theorem~\ref{tfthm} is a little easier, though,
Govaerts and Storme do prove a bit more; they
characterize the non-affine simple rank-$r$ triangle-free
binary matroids with $\frac{5}{16} 2^r$ elements.

\section{Connections with graph theory}

The following result is a weak version of Tur\' an's Theorem~[\ref{tur}].
\begin{theorem}
For all integers $t$ and $n$ with $n\ge t\ge 2$, if $G$ is a
simple $n$-vertex $K_t$-free graph, then
$|E(M)| \le \frac{t-2}{t-1} \binom{n}{2}$.
\end{theorem}
One natural class of $K_t$-free graphs is the
class of $(t-1)$-colourable graphs. The stronger
version of Tur\' an's Theorem amounts to saying that
the densest $K_t$-free graphs are all $(t-1)$-colourable.
One might hope that, for some $\epsilon>0$,
all $n$-vertex $K_t$-free graphs with
at least $\left(\frac{t-2}{t-1} -\epsilon\right)\binom n 2$ 
edges are $(t-1)$-colourable.
However, this is not true as one can take the 
direct sum of a triangle-free graph with chromatic number $t$
and some sufficiently large and dense graph with
chromatic number $t-1$.

Andr\' asfai, Erd\H os and S\' os~[\ref{aes}] overcome this
issue by considering minimum degree instead of the number of edges.
Note that, if $G$ is an $n$-vertex graph with minumum degree
$\alpha n$, then $|E(G)| > \alpha \binom n 2$.
\begin{theorem}[Andr\' asfai-Erd\H os-S\' os Theorem]
Let $t\ge 3$ be an integer and let $\epsilon = \frac{1}{(t-1)(3t-4)}$.
Then, for each integer $n\ge t$, if $G$ is a simple $n$-vertex
$K_t$-free graph with minimum degree
$> \left(\frac{t-2}{t-1} -\epsilon\right) n$, then
$G$ is $(t-1)$-colourable.
\end{theorem}

In some sense the geometric version is even nicer,
since the Geometric Andr\' asfai-Erd\H os-S\' os Theorem
implies the Bose-Burton Theorem, but 
it is not immediately evident whether or not the 
Andr\' asfai-Erd\H os-S\' os Theorem implies Tur\' an's Theorem.

The Andr\' asfai-Erd\H os-S\' os Theorem is proved by induction on $t$;
the base case is:
\begin{theorem}\label{base}
For each integer $n\ge 5$,
if $G$ is a simple $n$-vertex triangle-free graph with
minimum degree $> \frac 2 5 n$, then $G$ is bipartite.
\end{theorem}

They prove the following strengthening.
\begin{theorem}\label{oddgirth}
For each odd integer $k\ge 5$ and each integer $n\ge k$,
if $G$ is a simple $n$-vertex graph with
no odd-circuit of length less than $k$ and with
minimum degree $> \frac 2 k n$, then $G$ is bipartite.
\end{theorem}

The above results on graphs bear a striking resemblance
to the results in the introduction,
where the role of ``chromatic number"
in graphs replaces ``critical number" in the geometric setting.
It is well known that the critical number and 
the chromatic number are related.
For example, if $G$ is a simple graph of chromatic number $\chi$
and $M(G)$ has critical number $c\ge 1$, then
$$ 2^{c-1}< \chi \le 2^c.$$
In particular, $G$ is bipartite if and only if $M(G)$ is affine.
Moreover, the characterization of bipartite graphs using 
odd circuits is in fact a specialization of well-known
result about binary matroids; see Oxley~[\ref{oxley}, Proposition 9.4.1].
\begin{theorem}\label{affine}
A simple binary matroid is affine
if and only if it does not contain a circuit of odd size.
\end{theorem}

\section{The proofs}\label{proofs}

We start by proving Theorem~\ref{main},
which we reformulate here for convenience; the equivalence
between these formulations requires Theorem~\ref{affine}.
\begin{theorem}
Let $k\ge 5$ be an odd integer, 
let $M$ be a simple rank-$r$ binary matroid with $r\ge k-1$, and
let $C$ be a circuit of size $k$ in $M$.
If $M$ does not have an odd-circuit of length $<k$,
then $|M| \le \frac{k}{2^{k-1}} 2^r.$
\end{theorem}

\begin{proof}
We break the proof into three cases.

\medskip

\noindent
{\bf Case 1:}{\em\quad
$r = k-1$.}

\medskip

Suppose, for a contradiction, that 
$|M| > k = |C|$.  Let $e\in E(M) - E(C)$.
Since $M$ is simple and binary and since
$r(C)=r(M)$ there is a partition $(C_1,C_2)$ of 
$C$ with $|C_1|,|C_2|\ge 2$ such that
$C_1\cup\{e\}$ and $C_2\cup\{e\}$ are both
circuits. However, since $C$ is odd, one of
$C_1\cup\{e\}$ and $C_2\cup\{e\}$ is odd.
This contradicts that $C$  is a smallest odd circuit
in $M$.

\medskip

\noindent
{\bf Case 2:}{\em\quad
$r = k$.}

\medskip

Suppose, for a contradiction, that $|M| > |C| + k$.
By Case 1, $E(M) - C$ is a cocircuit of $M$.

\medskip

\noindent
{\bf Claim:}{\em \quad For each $u_1,u_2\in E(M)-C$ 
there exist $v_1,v_2\in C$ such that
$\{u_1,u_2,v_1,v_2\}$ is a circuit}

\medskip

Since $M|(C\cup \{u_1,u_2\})$ is binary and
has co-rank $2$, its ground set partitions
into three series classes $(\{u_1,u_2\},C_1,C_2)$.
Since $C$ is odd, we may assume that $|C_1|$ is odd.
Now $C_1\cup \{u_1,u_2\}$ is an odd circuit.
Since $C$ is an odd circuit of minimum size,
$|C_1|=|C|-2$ and, hence, $|C_2|=2$.
Now $C_2\cup\{u_1,u_2\}$ gives the required circuit.

\medskip

Let $u\in E(M) -C$ and let $X=E(M) -(C\cup\{u\})$.
By the claim, for each $e\in X$ there
exists a two-element set $P_e\subseteq C$ such that
$P_e\cup \{u,e\}$ is a circuit.
Moreover, since $M$ is binary,
$P_e\neq P_f$ for distinct $e,f\in X$.
Since $|X|\ge |C|$, there exist $e,f\in X$
such that $P_e$ and $P_f$ are disjoint.  Since $M$ is binary,
the symmetric difference $Z$ of
$C$, $P_e$, and $P_f$ can be partitioned
into circuits. However $Z$ is smaller
than $C$ and has odd size; this contradicts
that $C$ is a minumum sized odd-circuit.

\medskip

\noindent
{\bf Case 3:}{\em\quad
$r > k$.}

\medskip

By Claim 1, $\cl_M(C) = C$. 
By Claim 2, each parallel class of $M\con C$ has size 
at most $k$. Moreover, $M\con C$ has rank $r-k+1$ and
hence it has at most $2^{r-k+1}-1$ points. Therefore
$$ |M| \le k (2^{r-k+1} -1) + k = \frac{k}{2^{k-1}} 2^r, $$
as required.
\end{proof}

Now we prove the Geometric Andr\' asfai-Erd\H os-S\' os Theorem
from Theorem~\ref{tfthm}. 
This proof is sketched in~[\ref{gs}]; Govaerts and Storme
attribute the method to Beutelspacher~[\ref{beut}].
We reformulate the result here for convenience.
\begin{theorem} 
For all integers  $r$ and $n$ with $r-2\ge n \ge 2$,
if $M$ is a simple rank-$r$ $\PG(n-1,2)$-free binary matroid
with $|M| > \left(1-\frac{11}{2^{n+2}}\right) 2^r$,
then $M$ has critical number at most $n-1$.
\end{theorem}

\begin{proof}
Consider a counterexample $(r,n,M)$ with $n$ minimum.
Thus $M$ is a simple rank-$r$ $\PG(n-1,2)$-free binary matroid
with $|M|>\left(1-\frac{11}{2^{n+2}}\right) 2^r$ and
with critical number at least $n$.
By Theorem~\ref{tfthm}, $n\ge 3$.

Consider $M$ as a restriction of $\PG(r-1,2)$
and let $B$ denote the set of points not in $M$.
Thus $|B|  < \frac{11}{2^{n+2}}2^r -1$.

\medskip

\noindent
{\bf Claim 1:}{\em \quad
There is a line $l$ of $\PG(r-1,2)$ containing
exactly one point of $M$.}

\medskip

If not, then $B$ is a flat of $\PG(r-1,2)$.
Since $M$ has critical number at least $n$,
we have $r_M(B)\le r(M)-n$. So there 
is a rank-$n$ flat $F$ of $\PG(r-1,2)$ 
that is disjoint from $B$.
But then $M|F$ is isomorphic to $\PG(n-1,2)$.
This contradiction proves the claim.

\medskip

\noindent
{\bf Claim 2:}{\em \quad
There is a hyperplane $H$ of $\PG(r-1,2)$,
such that $|B\cap H|\ge 2^{r-n+1}-1$.}

\medskip

Let $l$ be a line containing exactly one point in $M$,
let $p\in l\cap E(M)$,
and let $H_0$ be a hyperplane of $M$ that does not contain $p$.
Let $X$ be the set of all points $q\in H_0\cap E(M)$
such that $\{p,q\}$ spans a triangle in $M$.
There are at most $2^{r-1}-2$ lines of $\PG(r-1,2)$
that contain $p$ and that contain at least one other point
of $M$. Each of these lines contains at most one point of 
$M\del (X\cup \{p\})$, so
$$ |M| \le 2^{r-1} -1 + |X|.$$
Thus $|X|> \left(1-\frac{11}{2^{n+1}}\right)2^{r-1}$.
Since $M$ is $\PG(n-1,2)$-free, $M|X$ is $\PG(n-2,2)$-free.
Therefore, by the minimality of the counterexample,
$M|X$ has critical number $\le n-2$.
Let $F_0$ be a rank-$(r-n+1)$ flat  in $H_0$ that is disjoint
from $X$ and let $F_1$ be the flat spanned by $F_0\cup \{p\}$.
By definition, $|F_1\cap B| \ge 2^{r-n+1}-1$.
We can extend $F_1$ to obtain the desired hyperplane;
this proves the claim.

\medskip

Let $H$ be a hyperplane satisfying Claim 2.

\medskip

\noindent
{\bf Claim 3:}{\em \quad
There is a rank-$(n-1)$ flat $F$ of $\PG(r-1,2)$
with $F\subseteq H\cap E(M)$.}

\medskip

Suppose otherwise; thus $M|(E(M)\cap H)$ is $\PG(n-2,2)$-free.
Since $M$ has critical number $\ge n$,
$M|(E(M)\cap H)$ has critical number $\ge n-1$.
Now, by the minimality of our counterexample,
$$ |E(M)\cap H| \le \left(1-\frac{11}{2^{n+1}}\right) 2^{r-1}.$$
Thus
$$ |M| \le |E(M)\cap H| + 2^{r-1} \le \left(1-\frac{11}{2^{n+2}}\right) 2^r,$$
giving the required contradiction. This proves the claim.

\medskip

Let $F$ be such a flat. There are 
$2^{r-n}$ flats of rank $n$ in $\PG(r-1,2)$
that contain $F$ but are not contained in $H$.
Since $M$ is $\PG(n-1,2)$-free,
each of these flats contains a point in $B$.
Thus $|B-H|\ge 2^{r-n}$. Therefore
$$
|B|  \ge   2^{r-n} + 2^{r-n+1} -1 
=\frac{12}{2^{n+2}} 2^r -1. 
$$
This contradiction completes the proof.
\end{proof}

\section{Extremal examples}\label{examples}

Our constructions are based on the following result.
\begin{lemma}\label{double}
Let $M$ be a simple rank-$r$ matroid, let $v\in E(M)$
such that each line containing $v$ has $3$ points,
and let $N$ be the restriction of $M$ to a hyperplane
not containing $v$. Then
\begin{itemize}
\item[(i)] $|M| = 2 |N| +1$.
\item[(ii)] $M\del v$ and $N$ have same critical number.
\item[(iii)] For each odd integer $k\ge 3$, if $M\del v$ has an
odd circuit of length $\le k$,
then $N$ has an odd circuit of length $\le k$.
\item[(iv)] For each integer $n\ge 2$, if $M\del v$ has a
$\PG(n-1,2)$-restriction, then $N$ has a $\PG(n-1,2)$-restriction.
\item[(v)] For each integer $n\ge 2$, if $N$ has a
$\PG(n-1,2)$-restriction, then $M$ has a $\PG(n,2)$-restriction.
\end{itemize}
\end{lemma}

Before we prove Lemma~\ref{double}, we introduce some definitions.
Note that $M$ is defined, up to isomorphism, from $N$.
We say that $M$ is a {\em conical lift} of $N$ and
that $M\del e$ is a {\em doubling} of $N$.

\begin{proof}[Proof of Lemma~\ref{double}.]
Note that (i) is trivial. Moreover,
since $\PG(n,2)$ is a conical lift of $\PG(n-1,2)$,
(v) is also trivial.

Consider $M$ as a restriction of $\PG(r-1,2)$ 
and let $H$ be the hyperplane of $\PG(r-1,2)$
containing $N$. Let $\bar N$ be the restriction
of $\PG(r-1,2)$ to $H-E(N)$ and let
$\bar M$ be the restriction of $\PG(r-1,2)$
to $(E(\PG(r-1,2) - E(M)) \cup \{v\}$.
Note that $\bar M$ is a conical lift of $\bar N$.
Hence (ii) follows from (v).

Now consider (iv). Suppose that $N_1$ is a restriction
of $M\del v$ that is isomorphic to $\PG(n-1,2)$.
Now $(M\con v)|E(N_1)$ is also isomorphic to
$\PG(n-1,2)$.
Since $N$ is isomorphic to the
simplification of $M\con v$,
$N$ has a restriction
isomorphic to $\PG(n-1,2)$, as required.

Finally, consider (iii).
Let $C$ be an odd circuit in $M\del v$.
We may assume that $C$ spans $v$ since otherwise
the proof goes as the proof of (iv).
Then there is an odd subset $C'$ of $C$ such that
$C'\cup \{v\}$ is a circuit in $M$.
Thus $C'$ is an odd circuit in $M'\con v$.
Since $N$ is isomorphic to the
simplification of $M\con v$, $N$ has an odd circuit
of length $|C'|\le k$.
\end{proof}

The following result shows that 
Theorem~\ref{main} is tight.
\begin{theorem}\label{maintight}
For each odd integer $k\ge 5$ and each integer $r\ge k-1$, 
there exists a non-affine rank-$r$ simple
$(\frac{k}{2^{k-1}} 2^r)$-element binary matroid with no
odd circuit of length less than $k$.
\end{theorem}

\begin{proof}
When $r=k-1$, we take the circuit of length $k$.
Then we construct examples in higher rank by repeatedly doubling.
\end{proof}

The next result shows that the 
Geometric Andr\' asfai-Erd\H os-S\' os Theorem is tight;
these examples were given in~[\ref{gs}].
\begin{theorem}\label{gsttight}
For all integers $n$ and $r$ with $r-2\ge n\ge 2$,  
there is a simple rank-$r$ $\PG(n-1,2)$-free binary matroid 
with critical number $n$ and
with $\left(1-\frac{11}{2^{n+2}} \right) 2^r$ elements.
\end{theorem}

\begin{proof}
For $n=2$, the examples come from 
Theorem~\ref{maintight}.
Suppose that $n\ge 3$ and that there
exists a
simple rank-$(r-1)$ $\PG(n-2,2)$-free binary matroid $N$
with $|N| = \left(1-\frac{11}{2^{n+1}} \right) 2^{r-1}$,
and with critical number $n-1$.
Let $H$ be a hyperplane in $\PG(r-1,2)$ and construct
a restriction $M$ of $\PG(r-1,2)$ by taking
a copy of $N$ in $H$ along with all points outside $H$.
Thus $M$ is $\PG(n-1,2)$-free, has critical number $n$, and has
$2^{r-1} +  \left(1-\frac{11}{2^{n+1}}\right) 2^{r-1}=
  \left(1-\frac{11}{2^{n+2}}\right) 2^{r}$ points.
\end{proof}

\section*{References}

\newcounter{refs}

\begin{list}{[\arabic{refs}]}%
{\usecounter{refs}\setlength{\leftmargin}{10mm}\setlength{\itemsep}{0mm}}

\item \label{aes}
B. Andr\' asfai, P. Erd\H os, V.T. S\' os,
On the connection between chromatic number, maximal clique
and minimum degree of a graph,
Discrete Math. 8 (1974) 205-218.

\item \label{beut}
A. Beutelspacher,
Blocking sets and partial spreads in finite projective spaces,
Geometriae Dedicata 9 (1980) 425-449.

\item \label{bb}
R.C. Bose, R.C. Burton,
A characterization of flat spaces in a finite geometry
and the uniqueness of the Hamming and the MacDonald codes,
J. Combin. Theory 1 (1966) 96-104.

\item \label{gs}
P. Govaerts, L. Storme,
The classification of the smallest nontrivial blocking sets in $\PG(n,2)$,
J. Combin. Theory Ser. A 113 (2006) 1543-1548.

\item\label{oxley}
J.G. Oxley,
{\em Matroid Theory}, second edition, Oxford University Press, New York, 2011.

\item\label{tur}
P. Tur\' an,
Eine extremalaufgabe aus der Graphentheorie,
Mat. \'es Fiz. Lapok 48 (1941) 436-452.

\end{list}

\end{document}